\documentclass[12pt]{article}
\usepackage{lineno,hyperref}
\modulolinenumbers[5]
\usepackage{color}

\usepackage{graphics,amsmath,amssymb}
\usepackage{amsthm}
\usepackage{amsfonts}
\usepackage{latexsym}
\usepackage{epsf}

\newcommand{\rsts}[2]{\genfrac{\{}{\}}{0pt}{0}{#1}{#2}}

\theoremstyle{plain}
\newtheorem{theorem}{Theorem}
\newtheorem{corollary}[theorem]{Corollary}

\theoremstyle{definition}

\theoremstyle{remark}
\newtheorem{remark}[theorem]{Remark}

\begin{document}

\begin{center}
\vskip 1cm{\LARGE\bf Integral Formulas for the Noncentral Tanny-Dowling Polynomials}
\vskip 1cm
\large    
Mahid M. Mangontarum$^1$, Norlailah M. Madid$^2$ and Asnawi A. Campong$^3$\\
Department of Mathematics\\
Mindanao State University-Main Campus\\
Marawi City 9700\\
Philippines \\
$^1$\href{mailto:mmangontarum@yahoo.com}{\tt mmangontarum@yahoo.com} \\
$^1$\href{mailto:mangontarum.mahid@msumain.edu.ph}{\tt mangontarum.mahid@msumain.edu.ph} \\
$^2$\href{mailto:norlailah.madid@msumain.edu.ph}{\tt norlailah.madid@msumain.edu.ph}\\
$^3$\href{mailto:campong.aa82@s.msumain.edu.ph}{\tt campong.aa82@s.msumain.edu.ph}
\end{center}

\begin{abstract}
	In this paper, we established some integral formulas for and involving the noncentral Tanny-Dowling polynomials. These formulas are shown to be generalizations of some known results on the classical geometric polynomials.
	
	\medskip
	
	
	\medskip
	
	
\end{abstract}

\section{Introduction}

Let $\rsts{n}{k}$ denote the Stirling numbers of the second kind \cite{Stirling}. In the classical distribution problems, $\rsts{n}{k}$ count the number of ways to distribute $n$ distinct objects into $k$ identical boxes such that no box is empty \cite[pg. 47]{Chen}. These numbers also appear as coefficients in the expansion of 
\begin{equation}
	x^n=\sum_{k=0}^n\rsts{n}{k}(x)_k,
\end{equation}
where $(x)_k=x(x-1)(x-2)\cdots(x-k+1)$ \cite{Comtet}. In relation to this, it is easy to show that the Stirling numbers of the second kind multiplied by the appropriate factorial $k!\rsts{n}{k}$ counts the number of ways to distribute $n$ distinct objects to $k$ distinct boxes such that no box is empty. On the other hand, the geometric polynomials (also known as Fubini Polynomials) \cite{Tanny}, denoted by $w_{n}(x)$, are defined to be the $n$th degree polynomial whose coefficients are $k!\rsts{n}{k}$. That is, 
\begin{equation}
	w_{n}(x)=\sum_{k=0}^{n}k!\rsts{n}{k}x^{k}. \label{eq:geopoly}
\end{equation}
These polynomials are known to satisfy the exponential generating function given by \cite[Eq. 3.14]{Boyadzhiev}
\begin{equation}
	\sum_{n=0}^{\infty}w_{n}(x)\frac{z^{n}}{n!}=\frac{1}{1-x(e^z-1)}.\label{eq:geopoly-egf}
\end{equation}
The case when $x=1$ given by 
\begin{equation}
	w_{n}:=w_{n}(1)=\sum_{k=0}^{n}k!\rsts{n}{k}\textcolor{red}{x^{k}}
\end{equation}
is called geometric numbers (of Fubini numbers). \textcolor{red}{These count all the possible set partitions of an $n$ element set such that the order of the blocks matters}. The exponential generating function of $w_n$ can be easily by setting $x=1$ in \eqref{eq:geopoly-egf}. That is,
\begin{equation}
	\sum_{n=0}^{\infty}w_n(1)\frac{z^n}{n!}:=\sum_{n=0}^{\infty}w_n\frac{z^n}{n!}=\frac{1}{2-e^z}.
\end{equation}

The study of geometric polynomials became a thrend for some mathematicians even up to this date. For instance, Kellner \cite{Kellner} established a number of identities involving the polynomials $w_n(x)$. Among these identities is the integral identity over the interval $[-1,0]$ given by
\begin{equation}
	\int_{-1}^{0} w_{n}(x)dx=B_{n}. \label{eq:Kellner'sid}
\end{equation}
Here, $B_{n}$ denotes the $n^{th}$ Bernoulli number defined by the exponential generating function 
\begin{equation}
	\sum_{n=0}^{\infty}B_{n}\frac{x^{k}}{n!}=\frac{x}{e^{x}-1}. \label{eq:bernoullinum}
\end{equation}
The proof of \eqref{eq:Kellner'sid} uses Worpitzky's identity \cite[pg. 215 (36)]{Worpitzky} given by
\begin{equation}
	B_n=\sum_{k=0}^{n}\sum_{j=0}^{k}(-1)^j\binom{k}{j}\frac{j^n}{k+1} \label{eq:bernoulli-exp}
\end{equation}
and its equivalent form
\begin{equation}
	B_{n}=\sum_{k=1}^{n}{(-1)}^{k}\frac{k!}{k+1}\rsts{n}{k}.\label{eq:bernoulli-exp2}
\end{equation}
Boyadzhiev \cite{Boyadzhiev} established transformation formulas for the geometric polynomials. In his paper, given the exponential polynomials (Bell polynomials) $\phi_{n}(x)$ defined by
\begin{equation}
	\phi_{n}(x)=\sum_{k=0}^{n}\rsts{n}{k}x^{k},
\end{equation}
Boyadzhiev \cite{Boyadzhiev} expressed the polynomials $w_{n}(x)$ in terms of the exponential polynomials, viz.
\begin{equation}
	w_{n}(x)=\int_{0}^{\infty} \phi_{n}(x\lambda)e^{-\lambda}d\lambda. \label{eq:geoexpo-rel}
\end{equation}
This was used to derive more properties for $w_{n}(x)$ including the exponential generating function \cite[Eq. (3.13)]{Boyadzhiev}
\begin{equation}
	\int_{0}^{\infty}e^{-\lambda(1-x(e^{x}-1))}d\lambda=\sum_{n=0}^{\infty}w_{n}(x)\frac{z^{n}}{n!}. \label{eq:geoexpo-egf}
\end{equation}
Other notable works are due to Karg{\i}n \cite{Kargin}, Dil and Kurt \cite{Dil1}, Boyadzhiev and Dil \cite{Boyad2}, Karg{\i}n and \c{C}ekim \cite{Kargin2}, and the references therein.

In 2016, Mangontarum et al. \cite{MCM-R} introduced the noncentral Tanny-Dowling polynomials $\widetilde{F}_{m,a}(n;x)$ defined by
\begin{equation}
	\widetilde{F}_{m,a}(n;x)=\sum_{k=0}^{n}k!\widetilde{W}_{m,a}(n,k)x^{k} \label{eq:noncentralTDpoly}
\end{equation}
and satisfying the exponential generating function
\begin{equation}
	\sum_{n=k}^{\infty}\widetilde{F}_{m,a}(n;x)\frac{z^{n}}{n!}=\frac{me^{-az}}{m-x(e^{mz}-1)}, \label{eq:noncentralTDpoly-egf}
\end{equation}
where the numbers $\widetilde{W}_{m,a}(n,k)$ are the noncentral Whitney numbers of the second kind, a generalization of $\rsts{n}{k}$. Further, in a recent paper by Mangontarum and Madid \cite{MM}, a number of identities for $\widetilde{F}_{m,a}(n;x)$ are established. Such identities are shown to generalizes some known results on the geometric polynomials, including the ones in the paper of Karg{\i}n \cite{Kargin}.

In this paper, we will establish integral formulas for and involving the noncentral Tanny-Dowling polynomials. In particulay, we will derive a generalization of Kellner's \cite{Kellner} integral formula relating the noncentral Tanny-Dowling polynomials with the Bernoulli polynomials, obtain generalizations of the Worpitzky's \cite{Worpitzky} explicit formulas in terms of noncentral Whitney numbers of the second kind, and derive a generalizations of Boyadzhiev's \cite{Boyadzhiev} identities for the noncentral Tanny-Dowling polynomials.

\section{Results and Discussions}

\begin{theorem}\label{Result1}
	For any real numbers $a$ and positive integer $m$, the following integral formula over the interval $[-1,0]$ holds:
	\begin{equation}
		\int_{-1}^{0}\widetilde{F}_{m,a}(n;mx)dx=m^nB_n\left(\frac{-a}{m}\right)
		\label{eq:Tanny-dowling1}
	\end{equation} 
\end{theorem}

\begin{proof}
Note that from \eqref{eq:noncentralTDpoly-egf}, we have
	\begin{equation}
		\sum_{n=0}^{\infty}\int_{-1}^{0}\widetilde{F}_{m,a}(n;mx)\frac{z^n}{n!}dx=\int_{-1}^{0}\frac{e^{-az}}{1-xe^{mz}+x}dx. \label{eq:laila}
	\end{equation}
Evaluating the integral and by \eqref{eq:bernoullinum}, we get
	\begin{align*}
		\sum_{n=0}^{\infty}\int_{-1}^{0}\widetilde{F}_{m,a}(n;mx)\frac{z^n}{n!}dx &=\frac{-e^{-az}}{e^{mz}-1}\ln\left|1-xe^{mz}+x\right|\biggl|^{0}_{1}\\
		&=\frac{-e^{-az}}{e^{mz}-1}\left(-\ln\left|e^{mz}\right|\right) \\
		&=\frac{mze^{-az}}{e^{mz}-1}\\
		&=\sum_{n=0}^{\infty}m^nB_n\left(\frac{-a}{m}\right)\frac{z^n}{n!}.
	\end{align*}
Comparing the coefficients of $\frac{z^n}{n!}$ completes the proof.
\end{proof}

\begin{remark}\label{rema1} \rm 
	Since $\widetilde{F}_{1,0}(n;x)=w_n(x)$, then by setting $m=1$ and $a=0$ in Theorem \ref{Result1}, we get the integral
	\begin{equation*}
		\int_{-1}^{0}\widetilde{F}_{1,0}(n;x)dx = B_n(0)
	\end{equation*}
	which is Kellner's \cite{Kellner} identity in \eqref{eq:Kellner'sid}.
\end{remark}

Now, observe that from \eqref{eq:noncentralTDpoly}, 
\begin{align*}
	m^nB_n\left(\frac{-a}{m}\right) &= \int_{-1}^{0} \left(\sum_{k=0}^{n}m^kk!\widetilde{W}_{m,a}(n,k)x^{k}\right)dx \\
	&= \sum_{k=0}^{n}m^kk!\widetilde{W}_{m,a}(n,k)\int_{-1}^{0}x^{k}dx \\
	&= \sum_{k=0}^{n}m^kk!\widetilde{W}_{m,a}(n,k)\frac{(-1)^k}{k+1}.
\end{align*}
Thus, we have the following corollary:

\begin{corollary}\label{Result2}
	The $n^{\rm th}$ Bernoulli polynomial $B_n\left(\frac{-a}{m}\right)$ satisfies the following explicit formula:
	\begin{equation}
		B_n\left(\frac{-a}{m}\right)=\sum_{k=0}^{n}k!\widetilde{W}_{m,a}(n,k)\frac{(-1)^k}{m^{n-k}(k+1)}.\label{result2}
	\end{equation}
\end{corollary}

\begin{remark}\rm 
	Since $\widetilde{W}_{1,0}(n,k)=\rsts{n}{k}$, then when $m=1$ and $a=0$ in Corollary \ref{Result2}, we recover the Bernoulli formula in \eqref{eq:bernoulli-exp2}. Moreover, using the explicit formula of $\widetilde{W}_{m,a}(n,k)$ \cite[(38)]{MCM-R} given by
	\begin{equation*}
		\widetilde{W}_{m,a}(n,k)=\frac{1}{m^kk!}\sum_{j=0}^{k}\binom{k}{j}(-1)^{k-j}(mj-a)^{n},
	\end{equation*}
	equation \eqref{result2} can be written as
	\begin{equation}
		B_n\left(\frac{-a}{m}\right)=\sum_{j=0}^{n}\sum_{j=0}^{k}\binom{k}{j}(mj-a)^n\frac{(-1)^j}{m^n(k+1)}.
	\end{equation}
	This is a generalization of Worpitzky's \cite{Worpitzky} identity in \eqref{eq:bernoulli-exp}.
\end{remark}

Before proceeding, note that by induction on $k$, it is easy to show that
\begin{equation}
	\int_{0}^{\infty}x^ke^{-x}dx=k!.\label{lemma}
\end{equation}
Also, we recall the noncentral Dowling polynomials \cite[Eq. (89)]{MCM-R} defined by
\begin{equation*}
	\widetilde{D}_{m,a}(n;x)=\sum_{k=0}^{n}\widetilde{W}_{m,a}(n,k)x^k
\end{equation*}
which satisfies the exponential generating function \cite[Eq. (91)]{MCM-R}
\begin{equation}
	\sum_{n=0}^{\infty}\widetilde{D}_{m,a}(n;x)\frac{z^n}{n!}=e^{-az+(emz-a)(x/m)}.\label{DowlingExp}
\end{equation}

The next theorem expresses the noncentral Tanny-Dowling polynomias in terms of an integral of the noncentral Dowling polynomials.
\begin{theorem}\label{Result3}
	The noncentral Tanny-Dowling polynomials satisfy the following relation:
	\begin{equation}
		\widetilde{F}_{m,a}(n;x)=\int_{0}^{\infty}\widetilde{D}_{m,a}(n;x\lambda)e^{-\lambda}\ d\lambda. \label{eq:result3}
	\end{equation}
\end{theorem}

\begin{proof}
By definition,
	\begin{align*}
		\int_{0}^{\infty} \widetilde{D}_{m,a}(n;x\lambda)e^{-\lambda}d\lambda &= \int_{0}^{\infty}\left[\sum_{k=0}^{n}\widetilde{W}_{m,a}(n,k)x^k\lambda^k\right] e^{-\lambda}d\lambda \\
		&= \sum_{k=0}^{n}\widetilde{W}_{m,a}(n,k)x^k\int_{0}^{\infty}\lambda^ke^{-\lambda}d\lambda.
	\end{align*}
Using \eqref{lemma} and then \eqref{eq:noncentralTDpoly} yield
	\begin{align*}
		\int_{0}^{\infty} \widetilde{D}_{m,a}(n;x\lambda)e^{-\lambda}d\lambda &= \sum_{k=0}^{n}k!\widetilde{W}_{m,a}(n,k)x^k\\
		&= \widetilde{F}_{m,a}(n;x)
	\end{align*}
which is the desired result.
\end{proof}

\begin{remark}\rm 
	Since $\widetilde{D}_{1,0}(n;x\lambda)=\phi_{n}(x\lambda)$, then when $m=1$ and $a=0$, the following relation 
	\begin{equation}
		\int_{0}^{\infty} \widetilde{D}_{1,0}(n;x\lambda)e^{-\lambda}\ d\lambda = \widetilde{F}_{1,0}(n;x)
	\end{equation} 
	is precisely Boyadzhiev's \cite{Boyadzhiev} formula in \eqref{eq:geoexpo-rel}.
\end{remark}

The next theorem presents another form of exponential generating function for the polynomials $\widetilde{F}_{m,a}(n;x)$.
\begin{theorem}\label{result4} 
	The exponential generating function of the noncentral Tanny-Dowling polynomial satisfies the following integral formula:
	\begin{equation}
		\sum_{n=0}^{\infty}\widetilde{F}_{m,a}(n;x)\frac{z^n}{n!}=\int_{0}^{\infty}exp\left[-az-\lambda(1-\frac{x}{m}(e^{mz}-1))\right]d\lambda.		\label{eq:result4}
	\end{equation}
\end{theorem}

\begin{proof}
	Multiplying both sides of \eqref{eq:result3} by $\frac{z^n}{n!}$ and summing over $n$ gives
	\begin{equation*}
		\sum_{n=0}^{\infty}\widetilde{F}_{m,a}(n;x)\frac{z^n}{n!}=\int_{0}^{\infty}\left(e^{-\lambda}\sum_{n=0}^{\infty}\widetilde{D}_{m,a}(n;x\lambda)\frac{z^n}{n!}\right)d\lambda.
	\end{equation*}
	By apply \eqref{DowlingExp} in the right-hand side,
	\begin{equation*}
		\sum_{n=0}^{\infty}\widetilde{F}_{m,a}(n;x)\frac{z^n}{n!}=\int_{0}^{\infty}e^{-az-\lambda(1-\frac{x}{m}(e^{mz}-1))}d\lambda.
	\end{equation*}
\end{proof}

\begin{remark}\rm 
	When $m=1$ and $a=0$, we obtain 
	\begin{equation}
		\int_{0}^{\infty}e^{-(0)z-\lambda(1-\frac{x}{1}(e^{(1)z}-1)}d\lambda = \sum_{n=0}^{\infty}\widetilde{F}_{1,0}(n;x)\frac{z^n}{n!},
	\end{equation}
	an equivalent representation of Boyadzhiev's \cite{Boyadzhiev} exponential generating function in \eqref{eq:geoexpo-egf}.
\end{remark}

\section*{Acknowledgment}

This research is funded by the Mathematical Society of the Philippines under the 2022 MSP Research Grants. It is also supported by the Mindanao State University by virtue of Special Order No. 624-OP, s. 2022.

\end{document}